\documentclass[12pt,twoside]{article}
\title{Category of $n$-weak injective and $n$-weak
	flat modules with respect to special super presented modules}
\date{}
\author{}
\usepackage{graphicx,fancyhdr,fancybox,rotating,xcolor,appendix}
\usepackage{ulem,hhline,dsfont,pifont,multicol,lmodern}
\usepackage{amsthm,amsbsy,geometry,times,pst-node}
\usepackage{amsmath,tikz-cd}

\usepackage{lipsum}
\usepackage{euscript}
\usepackage{hyperref}
\usepackage{hyperref}
\usepackage{hyperref}

\usepackage{textcomp}
\usepackage{pdfpages}
\usepackage{stmaryrd}
\usepackage[all]{xy}
\usepackage{amsfonts}
\usepackage{amsmath}
\usepackage{amssymb}
\usepackage{latexsym}
\usepackage{mathrsfs}

\newtheorem{thm}{Theorem}[section]
 \newtheorem{cor}[thm]{Corollary}
 \newtheorem{lem}[thm]{Lemma}
 \newtheorem{prop}[thm]{Proposition}
 \newtheorem{Def}[thm]{Definition}
\newtheorem{rem}[thm]{Remark}
 \newtheorem{ex}[thm]{Example}

\def\Mod{{\rm Mod}}

\newcommand{\X}{\rm \mathscr{X}}
\newcommand{\Y}{\rm \mathscr{Y}}

\newcommand{\F}{{\cal F}}

\begin{document}

\thispagestyle{empty}

\maketitle \vspace*{-1.5cm}
\begin{center}{\large\bf Mostafa Amini$^{1,a}$, Houda Amzil$^{2,b}$ and  Driss   Bennis$^{2.c}$ }
\bigskip

\small{1. Department of Mathematics, Faculty of Sciences, Payame Noor University, Tehran, Iran.\\
2. Department of Mathematics, Faculty of Sciences, Mohammed V University in Rabat, Morocco.\\
$\mathbf{a.}$ amini.pnu1356@gmail.com \\
$\mathbf{b.}$ houda.amzil@um5r.ac.ma; ha015@inlumine.ual.es \\
$\mathbf{c.}$  driss.bennis@um5.ac.ma; driss$\_$bennis@hotmail.com \\}

\small{$\mathbf{}$  
}
\end{center}

\bigskip

\noindent{\large\bf Abstract.} Let $R$ be a  ring and $n$, $k$ two non-negative integers. In this paper,  we introduce the concepts of $n$-weak injective and $n$-weak flat modules and via the
notion of special super finitely presented modules, we obtain some  characterizations of these modules. We also investigate two classes of
modules with richer contents, namely  $\mathcal{WI}_k^n(R)$ and $\mathcal{WF}_k^n(R^{op})$ which are larger than that of modules  with  weak injective
and weak flat dimensions at most $k$. Then over any arbitrary ring, we study the existence of $\mathcal{WI}_k^n(R)$ and $\mathcal{WF}_k^n(R^{op})$ covers and
preenvelopes. \bigskip

\medskip

\bigskip\bigskip
%

\section{Introduction and Preliminaries}
\ \ \  Injectivity of modules is one of the principal notions in homological algebra. Namely, over Noetherian rings, injective modules have very important properties as well as many applications since the classical Matlis' work (see \cite{Matlis}). Stenstr\"{o}m introduced FP-injective modules and studied it over coherent rings as a counterpart to injective modules over Noetherian rings (see \cite{Su}).  Recall that a left
$R$-module $M$ is called {\it FP-injective} if ${\rm Ext}_{R}^{1}(U, M)=0$ for any finitely presented left
$R$-module $U$. Accordingly, the FP-injective dimension of $M$, denoted by FP-${\rm id}_{R}(M)$, is defined to be the smallest $n\geq 0$ such that ${\rm Ext}_{R}^{n+1}(U, M)=0$ for all finitely presented
left $R$-modules $U$. If no such $n$ exists, one defines FP-${\rm id}_{R}(M)=\infty$.  For background on
FP-injective (or absolutely pure) modules, we refer the reader to \cite{EM,BH,LM1,LD,BC,QD,PK,Su}. 

 Recall that  coherent rings first appeared in Chase's paper \cite{Chase} without being mentioned by name. 
The term coherent was first used by Bourbaki in \cite{Bou}. Then,   $n$-coherent rings  were introduced by Costa in \cite{Costa}. Let $n$ be a non-negative integer.  A left $R$-module   $M$ is said to be {\it $n$-presented} if there is an exact sequence 
 $ F_{n}\rightarrow F_{n-1}\rightarrow\dots\rightarrow F_1\rightarrow F_0\rightarrow
M\rightarrow$ of   left $R$-modules, where each $F_i$ is finitely generated and free. And a ring $R$ is called left {\it $n$-coherent} if every
$n$-presented  $R$-module is $(n + 1)$-presented. Thus,  for $n=1$,  left $n$-coherent rings  are nothing but left coherent rings (see \cite{Costa, DKM, Su}). 

In 2015, Wei and Zhang in \cite{JZW}, introduced the notion of $fp_{n}$-injective modules as a generalization of the notion of FP-injective modules by using $n$-presented modules. They also introduce $fp_{n}$-flat modules. Namely, a left $R$-module $M$ is said to be {\it $fp_{n}$-injective} if for every exact sequence $0\rightarrow X\rightarrow Y$ with $X$ and $Y$ $n$-presented left $R$-modules, the induced sequence ${\rm Hom}(Y,M)\rightarrow {\rm Hom}(X,M)\rightarrow 0$ is exact. And a right $R$-module $N$ is said to be {\it $fp_{n}$-flat} if for every exact sequence $0\rightarrow X\rightarrow Y$ with
$X$ and $Y$ are $n$-presented left $R$-modules, the induced sequence $0\rightarrow N\otimes_{R}X\rightarrow N\otimes_{R}Y$ is exact. They investigated the properties of these modules and in particular, proved the existence of $fp_{n}I$-covers (respectively preenvelopes) and $fp_{n}F$-covers (respectively preenvelopes), where $fp_{n}I$ and $fp_{n}F$ denote respectively the classes of $fp_{n}$-injective and $fp_{n}$-flat modules (see \cite[Theorem 2.5]{JZW}).  On the other hand, Bravo et al. in  \cite{BGH} introduced the notion  of absolutely clean and level modules, which Gao and Wang in \cite{Z.W} named weak injective and weak flat, respectively.    

In this paper, we deal with weak injective and weak flat modules and some extensions of these notions. 
 In 2015, Gao and Wang in \cite{Z.W}, using  super finitely presented modules instead of finitely presented modules \cite{Z.G}, introduced the notion  of
weak injective and weak flat modules. In general,  weak injective and weak flat modules are generalizations of $fp_n$-injective and $fp_n$-flat modules, respectively. A left $R$-module $U$ is called {\it super finitely presented} \cite{Z.G} if there exists an exact sequence
$ \cdots\rightarrow  F_{2}\rightarrow  F_1\rightarrow  F_0\rightarrow  U\rightarrow  0$, where each $F_i$
is finitely generated and free.  A left $R$-module $M$ is called  {\it weak injective} if ${\rm Ext}_{R}^{1}(U, M)=0$ for any super finitely presented left $R$-module $U$. A right  $R$-module $M$ is called  {\it weak flat} if ${\rm Tor}_{1}^{R}(M, U)=0$ for any super finitely presented left $R$-module $U$. 

Using weak injective and weak flat modules, several authors investigated the homological aspect of some notions over arbitrary rings, generalizing by this some known results over coherent rings. For example, in 2018, Zhao in \cite{.NG} investigated the homological aspect of modules with finite
weak injective and weak flat dimensions. Namely, if $\mathcal{WI}_k(R)$ and $\mathcal{WF}_k(R^{op})$  are the classes of left and right modules with 
weak injective dimension and weak flat dimensions at most $k$, respectively,  then by using derived functors ${\rm Ext}^{\mathcal{WF}}$, ${\rm Ext}^{\mathcal{WI}}$ and ${\rm Tor}_{\mathcal{W}}$ on $\mathcal{WF}(R^{op})$-resolutions and $\mathcal{WI}(R)$-resolutions, it was proved that the classes $\mathcal{WI}_k(R)$ and $\mathcal{WF}_k(R^{op})$ are  covering and preenveloping over any arbitrary ring, where  $\mathcal{WI}(R)$ and $\mathcal{WF}(R^{op})$ are the class of weak injective left modules and weak flat right modules, respectively. When $k=0$ and for coherent rings, every module has an FP-injective cover and an FP-injective preenvelope. In the recent years, the homological theory for weak injective and weak flat modules 
 has become an important area of research (see \cite{Z.H,Z.W,TX}).

Let $n,k$ be  non-negative integers. In this paper, we introduce the concept of $n$-weak injective left modules and $n$-weak flat right modules  by using $n$-super finitely presented  left modules. Every $n$-weak injective (resp. $n$-weak flat) is weak injective (resp. weak flat). And if $n\geq 1$, then $n$-weak injective (resp. $n$-weak flat) modules  are  common generalizations of  weak injective and  $fp_n$-injective (resp. weak flat and $fp_n$-flat) modules. Under these definitions, $n$-super finitely presented,  $n$-weak injective and $n$-weak flat modules are weaker than the usual super finitely presented,  weak injective (resp. $fp_n$-injective) and weak flat (resp. $fp_n$-flat) modules, respectively.
Also, for any  $m\geq n$, every $n$-super finitely presented, $n$-weak injective (resp. $fp_n$-injective) and $n$-weak flat (resp. $fp_n$-flat) modules is $m$-super finitely presented, $m$-weak injective and  $m$-weak flat, respectively. But, $m$-weak injective and  $m$-weak flat $R$-modules need not be  $n$-weak injective (resp. $fp_n$-injective) and $n$-weak flat (resp. $fp_n$-flat) for any $n<m$ (resp. $n\leq m$) (see Example \ref{1.3a}). We also introduce and investigate the classes 
$\mathcal{WI}_k^n(R)$ and $\mathcal{WF}_k^n(R^{op})$ which are larger than the classes $\mathcal{WI}_k(R)$ and $\mathcal{WF}_k(R^{op})$ (resp. $fp_{n}I$ and $fp_{n}F$)  in \cite{.NG} (resp. \cite{JZW}).

The paper is organized as follows:

In Sec. 1, some fundamental notions and some preliminary results are stated.

In Sec. 2, we introduce $n$-super finitely presented, $n$-weak injective left modules and $n$-weak flat right modules. Then, by considering special super finitey presented  $R$-modules of every $n$-super finitely presented left $R$-module, we give some characterizations of these modules. 

In Sec. 3, we give some homological aspects of modules with finite $n$-weak injective and $n$-weak flat dimensions. We let $\mathcal{WI}_{k}^n(R)$ and $\mathcal{WF}_k^n(R^{op})$  denote the classes  of left and right modules with 
$n$-weak injective dimension and $n$-weak flat dimension at most $k$, respectively. Among other results, we prove that over an arbitrary ring, $M$ is in $\mathcal{WF}_k^n(R^{op})$ (resp. $\mathcal{WI}_{k}^n(R)$) if and only if $M^*$ is in $\mathcal{WI}_{k}^n(R)$ (resp. $\mathcal{WF}_k^n(R^{op})$), where $M^{*}={\rm Hom}_{\mathbb{Z}}(M,{\mathbb{Q}}/{\mathbb { Z}})$. Also, considering an exact sequence  $0\rightarrow A\rightarrow B\rightarrow C\rightarrow 0$ of $R$-modules, we show that if $A$ and $ B$ are in $\mathcal{WI}_{ k}^{n}(R)$ then  $C$ is in $\mathcal{WI}_{ k}^{n}(R)$ and if $B$ and $C$ are in $\mathcal{WF}_{ k}^{n}(R^{op})$ then $A$ is in $\mathcal{WF}_{ k}^{n}(R^{op})$.

In  Sec. 4, we show that over an  arbitrary ring, $\mathcal{WI}_{k}^n(R)$ and $\mathcal{WF}_k^n(R^{op})$  are  injectively resolving and  projectively resolving and consequently,
 we prove that the classes  $\mathcal{WI}_{k}^n(R)$ and $\mathcal{WF}_k^n(R^{op})$ are  covering and preenveloping. Then by considering $n=0$,  we deduce that the classes  $\mathcal{WI}_{k}(R)$ and $\mathcal{WF}_k(R^{op})$ are  covering and preenveloping (\cite[Theorems 4.4, 4.5, 4.8 and 4.9]{.NG}). Moreover, if $k=0$, then Theorems 2.5 and 2.7 of \cite{JZW} follow. Also, we  investigate rings over which every left module is in $\mathcal{WI}^n(R)$  and every right module is in $\mathcal{WF}^n(R^{op})$. Finally, we show that the pair $(\mathcal{WF}_{k}^n(R^{op}), \mathcal{WF}_{k}^n(R^{op})^{\bot})$ is a hereditary perfect cotorsion pair, and if $R$ is in $\mathcal{WI}_{k}^n(R)$, it follows that the pair $(\mathcal{WI}_{k}^n(R), \mathcal{WI}_{k}^n(R)^{\bot})$ 
is a perfect cotorsion pair.
\section{$n$-Weak injective and $n$-weak flat  modules}
\ \ \  In this section, we introduce the notions of $n$-weak injective and $n$-weak flat modules using special super finitely presented modules. Then, we show some of their general properties. We start with the following definition.

\begin{Def}\label{1.1} 
Let  $n$ be a non-negative integer. A  left $R$-module $U$ is said to be $n$-super finitely presented if there exists an exact sequence $$ \cdots\rightarrow  F_{n+1}\rightarrow  F_{n}\rightarrow \cdots\rightarrow  F_1\rightarrow  F_0\rightarrow  U\rightarrow  0$$ of projective $R$-modules, where each $F_i$ is  finitely generated and projective for any $i\geq n$.

 If $K_{i}:={\rm Im}(F_{i+1}\rightarrow F_{i})$, then for $i=n-1$, we  call the module $K_{n-1}$ special super finitely presented left $R$-module. Moreover, if ${\rm Hom}_{R}(K_{n-1},-)$ is exact with respect to a short exact sequence 
$0\rightarrow A\rightarrow B\rightarrow C\rightarrow 0$ of left $R$-modules, then we say that this sequence is special superpure and $A$ is said to be superpure in $B$.
\end{Def}

\begin{Def}\label{1.2} 
Let  $n$ be a non-negative integer. A left $R$-module $M$ is called $n$-weak injective if ${\rm Ext}_{R}^{n+1}(U,M)=0$
 for every $n$-super finitely presented left $R$-module $ U$. A right $R$-module $N$ is called $n$-weak flat
if ${\rm Tor}_{n+1}^{R}(N,U)=0$ for every $n$-super finitely presented left $R$-module $U$.
\end{Def}
\begin{rem}\label{1.3}
Let  $n,m,k$  be non-negative integers. Then:
\begin{enumerate}
\item [\rm (1)]
${\rm Ext}_{R}^{n+1}(U,-)\cong{\rm Ext}_{R}^{1}(K_{n-1},-)$ and  ${\rm Tor}_{n+1}^{R}(-,U)\cong {\rm Tor}_{1}^{R}(-,K_{n-1})$, where $U$ is an $n$-super finitely presented left $R$-module with a special super finitely presented module $K_{n-1}$.  If $n=0$, then $n$-weak injective left $R$-modules, $n$-super finitely presented  left $R$-modules and $n$-weak flat right $R$-modules are simply weak injective left $R$-modules, super finitely presented left $R$-module and weak flat right $R$-module, respectively.
\item [\rm (2)]
 Every super finitely presented left $R$-module is $n$-super finitely presented.
\item [\rm (3)]
 Every $n$-super finitely presented left $R$-module is $m$-super finitely presented for any $m\geq n$,  but not conversely (see Examples \ref{1.1a} and \ref{1.3a}(1)). If we denote by ${\rm Pres}_{n}^{\infty}$ the class of all $n$-super finitely presented
 left $R$-modules, then:
$${\rm Pres}_{n}^{\infty}\subseteq{\rm Pres}_{n+1}^{\infty}\subseteq {\rm Pres}_{n+2}^{\infty}\subseteq\cdots$$
If $n=0$, then ${\rm Pres}_{0}^{\infty}$ is simply the class of all super finitely presented  left $R$-modules. We denote this class simply by ${\rm Pres}^{\infty}$.
\item [\rm (4)]
Every $n$-weak injective left (resp. $n$-weak flat right) $R$-module is $m$-weak injective (resp. $m$-weak flat) for any $n\leq m$, but not conversely (see Example \ref{1.3a}(2)). If $U$ is an $(n+1)$-super finitely presented left $R$-module, then there exists an exact sequence
$$ \cdots\rightarrow  F_{2}\rightarrow  F_1\rightarrow  F_0\rightarrow  U\rightarrow  0,$$ where $K_{n}$ is a special super finitely presented left module. Also, we have the short exact sequence $0\rightarrow K_{0}\rightarrow F_0\rightarrow U\rightarrow 0$, where $K_0$ is an $n$-super finitely presented left module. So, if $M$ is an $n$-weak injective left $R$-module, then ${\rm Ext}_{R}^{n+1}(K_0, M)=0$. On the other hand, ${\rm Ext}_{R}^{n+2}(U,M)\cong{\rm Ext}_{R}^{n+1}(K_{0}, M)=0$, and hence $M$ is $(n+1)$-weak injective. Similarly, every $n$-weak flat right $R$-module is $(n+1)$-weak flat. 
\item [\rm (5)]
If $\mathcal{I}$, $\mathcal{FP}$, $\mathcal{WI}(R)$, $\mathcal{WI}^{n}(R)$,  $\mathcal{F}$, $\mathcal{WF}(R^{op})$ and $\mathcal{WF}^n(R^{op})$ are the classes of  injective, FP-injective, weak injective, $n$-weak injective left $R$-modules and flat, weak flat and $n$-weak flat  right  $R$-modules, respectively, then

\ \ \ \ \ \ \ \ \ \ \ \ \ \ \ \ \ \ \ \ \ \ $\mathcal{I}\subseteq\mathcal{FP}\subseteq\mathcal{WI}(R)\subseteq\mathcal{WI}^{n}(R)\subseteq\mathcal{WI}^{n+1}(R)\subseteq\mathcal{WI}^{n+2}(R)\subseteq\cdots$\\ and 

  \ \ \ \ \ \ \ \ \ \ \ \ \ \ \ \ \ \ \ \ \ \ $\mathcal{F}\subseteq\mathcal{WF}(R^{op})\subseteq\mathcal{WF}^{n}(R^{op})\subseteq\mathcal{WF}^{n+1}(R^{op})\subseteq\mathcal{WF}^{n+2}(R^{op})\subseteq\cdots$.
\item [\rm (6)] Every $fp_n$-injective left $R$-module is $n$-weak injective and every $fp_n$-flat right $R$-module is $n$-weak flat. Indeed, for any 
$n$-super finitely presented left module $U$, there exists a short exact sequence $0\rightarrow K_{n}\rightarrow F_{n}\rightarrow K_{n-1}\rightarrow 0$, where $K_{n-1}$ is a special super finitely presented left module. So if $M$ is $fp_n$-injective left $R$-module, then ${\rm Hom}(F_n,M)\rightarrow {\rm Hom}(K_n,M)\rightarrow 0$ is exact, since  $K_{n}$ and $F_{n}$ are super finitely presented and consequently are $n$-presented. Therefore, ${\rm Ext}_{R}^{1}(K_{n-1}, M)=0$ and thus by (1), it follows that $M$ is $n$-weak injective. Similarly, every  $fp_n$-flat right $R$-module is $n$-weak flat, but not conversely (see Example \ref{1.3a}(3)).
\end{enumerate}
\end{rem}
Let $A$ be an $R$-module. Then, the
finitely presented dimension of $A$ denoted by ${\rm f.p.dim}_{R}(A)$ is defined as 
${\rm inf}\{n \ \mid \ {\rm there \ exists \ an \ exact \
sequence }\  {\rm F}_{n+1}\rightarrow {\rm F}_{n}\rightarrow\cdots \rightarrow {\rm F}_1\rightarrow {\rm F}_0\rightarrow {\rm A}\rightarrow 0 \ {\rm of} \ R$-${\rm modules, \ where \ each  \  F_{i}  \ projective, \ and \ F_{n} \ and \  F_{n+1} \ are \ finitely\  generated} \}$.
So ${\rm f.p.dim}(R)$ = ${\rm sup}\{{\rm f.p.dim}_{R}(A) \ \mid \ {\rm A \ is \ a\ finitely\  generated} \ R$-${\rm module} \}$. 
 We use ${\rm w.gl.dim}(R)$ and ${\rm gl.dim}(R)$ to denote the weak global dimension and global dimension of a ring $R$ respectively. 
Also, a ring $R$ is called an $(a,b,c)$-ring,  if ${\rm w.gl.dim}(R)=a$, ${\rm gl.dim}(R)=b$ and ${\rm f.p.dim}(R)=c$ (see \cite{HKN}).
\begin{ex}\label{1.1a}
Let $R=k[x_1, x_2]\oplus R^{'}$, where $k[x_1,x_2]$ is a ring of polynomials in $2$ indeterminate over a field $k$, and $R^{'}$ is a valuation ring with global
dimension $2$. Then by \cite[Proposition 3.10]{HKN}, $R$  is a coherent $(2,2,3)$-ring. So  ${\rm f.p.dim}(R)=3$, and hence there exists a finitely generated $R$-module $U$  such that ${\rm f.p.dim}_{R}(U)=3$. Thus, there exists an exact sequence
$F_{4}\rightarrow  F_{3}\rightarrow F_2\rightarrow F_1\rightarrow F_0\rightarrow U\rightarrow  0,$
where $F_3$ and $F_{4}$ are finitely generated and projective modules. Also, $K_2:={\rm Im}( F_3\rightarrow F_2)$ is special super finitely presented, since $R$ is coherent. So by Definition \ref{1.1}, $U$ is $3$-super finitely presented. But, $U$ is not $2$-super finitely presented otherwise ${\rm f.p.dim}_{R}(U)=2$, a contradiction.
\end{ex}
\begin{ex}\label{1.3a}
Let $R=k[x_1, x_2]\oplus S$, where $k[x_1,x_2]$ is a ring of polynomials in $2$ indeterminates over a field $k$, and $S$ is a
non-Noetherian hereditary von Neumann regular ring  (for example $S$ is a ring of functions of $X$ into $k$ continuous with respect to the discrete topology on $k$, where $k$ is a field and $X$ is a totally disconnected compact Hausdorff space whose associated Boolean ring is hereditary, see examples of \cite{GMB}).
 Then by \cite[Proposition 3.8]{HKN}, $R$  is a coherent $(2,2,2)$-ring. Hence,
\begin{enumerate}
\item [\rm (1)]
Since ${\rm f.p.dim}(R)=2$, then by \cite[Proposition 1.5]{HKN},  for a finitely generated $R$-module $U$ either ${\rm f.p.dim}_{R}(U)=2$ or ${\rm f.p.dim}_{R}(U)=0$. If ${\rm f.p.dim}_{R}(U)= 0$, it follows that $U$ is  finitely presented. If  ${\rm f.p.dim}_{R}(U)=2$, then there exists an exact sequence
$F_{3}\rightarrow  F_{2}\rightarrow F_1\rightarrow F_0\rightarrow U\rightarrow 0,$
where $F_2$ and $F_{3}$ are finitely generated and projective $R$-modules. Also, $K_1:={\rm Im}( F_2\rightarrow F_1)$ is a special super finitely presented module, since $R$ is coherent. So by Definition \ref{1.1}, $U$ is $2$-super finitely presented. But, $U$ is neither $1$-super finitely presented  nor $0$-super finitely presented, otherwise every  finitely generated $R$-module would be  finitely presented, and hence by \cite[Theorem 1.3]{HKN}, ${\rm f.p.dim}(R)=0$,  a  contradiction.
\item [\rm (2)]
It is clear that every $R$-module is $2$-weak injective, since ${\rm gl.dim}(R)=2$. But not every $R$-module is $0$-weak injective. Indeed, if every module $M$ is $0$-weak injective, then ${\rm Ext}_{R}^{1}(U^{'}, M)=0$ for each $0$-super finitely presented $R$-module $U^{'}$. So each $0$-super finitely presented module is projective. But since
 $R$ is coherent, any finitely presented module is super finitely presented, and so any finitely presented module is projective. Hence $R$ is $1$-regular, and then by \cite[Theorem 3.9]{.TZ}, every $R$-module is flat. So ${\rm w.gl.dim}(R)=0$,  a  contradiction. Similarly, since ${\rm w.gl.dim}(R)=2$, it follows that every $R$-module is $2$-weak flat, but not every $R$-module $0$-weak flat.
\item [\rm (3)]
Not every $R$-module is ${\rm fp}_{2}$-injective (resp. ${\rm fp}_{2}$-flat). Indeed, suppose otherwise.  Since $R$ is coherent, it follows that any finitely presented  $R$-module $A$  is super finitely presented. Then, there is an exact sequence $0\rightarrow L_0\rightarrow F_0\rightarrow A\rightarrow 0,$ where $F_0$ and $L_0$ are $n$-presented $R$-modules, and so are also $2$-presented.
So, if $M$ is an ${\rm fp}_{2}$-injective $R$-module, then ${\rm Ext}_{R}^{1}(A, M)=0$ (resp. ${\rm Tor}_{1}^{R} (M, A)=0$) and consequently, $A$ is projective. Hence, $R$ is $1$-regular and thus ${\rm w.gl.dim}(R)=0$,  a  contradiction.
\end{enumerate}
\end{ex}

We denote by $R$-$\Mod$ the category of left $R$-modules and by $\Mod$-$R$ that of right $R$-modules.
\begin{prop}\label{1.4}
Let $n$ be a non-negative integer. Then, the following assertions hold:
\begin{enumerate}
\item [\rm (1)] 
 For every $M\in\Mod$-$R$, $M$ is $n$-weak flat if and only  if  $M^{*}$ is $n$-weak injective.
\item [\rm (2)] 
For every $M\in R$-$\Mod$, $M$ is $n$-weak injective if and only  if  $M^{*}$ is $n$-weak flat.
\end{enumerate}
\end{prop}
\begin{proof}
(1) Let $U$ be an $n$-super finitely presented left $R$-module with special super finitely presented module $K_{n-1}$.  Then, ${\rm Tor}_{1}^{R}(M,K_{n-1})^{*}\cong{\rm Ext}_{R}^{1}(K_{n-1}, M^*)$, since by \cite[Theorem 2.76]{Rot2}, 
${\rm Hom}_{\mathbb{Z}}(M\otimes_{R}K_{n-1},{\mathbb{Q}}/{\mathbb{Z}})\cong{\rm Hom}_{R}(K_{n-1},{\rm Hom}_{\mathbb{Z}}(M,{\mathbb{Q}}/{\mathbb{Z}}))$. The result follows from  Remark \ref{1.3}(1), since ${\rm Ext}_{R}^{n+1}(U,M^*)\cong{\rm Ext}_{R}^{1}(K_{n-1}, M^*)\cong{\rm Tor}_{1}^{R}(M,K_{n-1})^{*}\cong{\rm Tor}_{n+1}^{R}(M,U)^{*}$.

(2) Let $U$ be an $n$-super finitely presented left $R$-module with  special super finitely presented module $K_{n-1}$. Then, ${\rm Ext}_{R}^{1}(K_{n-1}, M)^{*}\cong{\rm Tor}_{1}^{R}(M^*,K_{n-1})$, since  by \cite[Lemma 3.55]{Rot2}, we have that $M^{*}\otimes_{R}K_{n-1}\cong{\rm Hom}_{R}(K_{n-1}, M)^*$. The result follows from  Remark \ref{1.3}(1), since ${\rm Tor}_{n+1}^{R}(M^*,U)\cong{\rm Tor}_{1}^{R}(M^*,K_{n-1})\cong{\rm Ext}_{R}^{1}(K_{n-1}, M)^{*}\cong{\rm Ext}_{R}^{n+1}(U, M)^{*}.$ \end{proof}

As a direct consequence of Proposition \ref{1.4} we obtain the following corollary.
\begin{cor}\label{1.4a}
Let $M$ be an  $R$-module. Then,
\begin{enumerate}
\item [\rm (1)] 
For every $M\in R$-$\Mod$, $M$ is $n$-weak injective if and only  if  $M^{**}$ is $n$-weak injective.
\item [\rm (2)] 
For every $M\in\Mod$-$R$, $M$ is $n$-weak flat if and only  if  $M^{**}$ is $n$-weak flat.
\end{enumerate}
\end{cor}

\begin{prop}\label{1.10}
Let $M$ be a left $R$-module. Then, the following assertions are equivalent:
\begin{enumerate}
\item [\rm (1)] 
 $M$  is $n$-weak injective.
\item [\rm (2)] 
 Every short exact sequence $0\rightarrow M\rightarrow B\rightarrow C\rightarrow 0$ of left $R$-modules  is  special superpure.
\item [\rm (3)] 
$M$ is special  superpure in any  left $R$-module containing it.
\item [\rm (4)] 
$M$ is special  superpure in any injective  left $R$-module containing it.
\item [\rm (5)]
$M$ is special  superpure in $E(M)$.
\end{enumerate}
\end{prop}
\begin{proof}
$(1)\Rightarrow  (2)$ Let $U$ be an $n$-super finitely presented left $R$-module with  special super finitely presented module $K_{n-1}$. Then by Remark \ref{1.3}(1), ${\rm Ext}_{R}^{n+1}(U,M)\cong{\rm Ext}_{R}^{1}(K_{n-1},M)=0$. Consequently, ${\rm Hom}_{R}(K_{n-1}, -)$ is exact with respect to the  short exact sequence $0\rightarrow M\rightarrow B\rightarrow C\rightarrow 0$.

$(2)\Rightarrow  (3)\Rightarrow  (4) \Rightarrow  (5)$ Clear.

$(5)\Rightarrow  (1)$ The  short exact sequence $0\rightarrow M\rightarrow E(M)\rightarrow {E(M)}/{M}\rightarrow 0$  is  special superpure. Therefore, if $U$ is an $n$-super finitely presented left $R$-module with  special super finitely presented module $K_{n-1}$, then by assumption and Remark \ref{1.3}(1),
$0={\rm Ext}_{R}^{1}(K_{n-1},M)\cong{\rm Ext}_{R}^{n+1}(U,M)$ and hence $M$  is $n$-weak injective.
\end{proof}
\begin{prop}\label{1.5}
	
	 Let $n$ be a non-negative integer. Then, the following assertions hold:
\begin{enumerate}
	\item [\rm (1)]  Let $\{M_{i}\}_{i\in I}$ be a family of left $R$-modules. Then
	$\prod_{i\in I} M_{i}$ is $n$-weak injective if and only if each $M_i$ is  $n$-weak injective.
	\item  [\rm (2)]  Let $\{M_{i}\}_{i\in I}$ be a family of right $R$-modules. Then
	$\bigoplus_{i\in I} M_{i}$ is $n$-weak flat if and only if each $M_i$ is  $n$-weak flat.
\end{enumerate}
\end{prop}

\begin{proof} Both assertions follow easily from \cite[Propositions 7.6 and 7.22]{Rot2}.
\end{proof}

\begin{prop}\label{1.5a}
Let $n$ be a non-negative integer and let $\{M_{i}\}_{i\in I}$ be a family of left $R$-modules. Then
$\bigoplus_{i\in I} M_{i}$ is $n$-weak injective if and only if each $M_i$ is  $n$-weak injective.
\end{prop}

\begin{proof}
 If $U$ is an $n$-super finitely presented left $R$-module with  special super finitely presented module $K_{n-1}$, then there exists a short exact sequence $0\rightarrow K_{n}\rightarrow F_{n}\rightarrow K_{n-1}\rightarrow 0$ of left $R$-modules, where  $F_{n}$ is finitely generated and projective and $K_{n}$ is super finitely presented. Thus we have the following commutative diagram:

$$\xymatrix{
{\rm Hom}_{R}(F_{n}, \bigoplus_{i\in I}M_{i})\ar[r]\ar[d]^{\cong}&{\rm Hom}_{R}(K_{n}, \bigoplus_{i\in I} M_{i})\ar[r]\ar[d]^{\cong}&{\rm Ext}_{R}^{1}(K_{n-1}, \bigoplus_{i\in I} M_{i}) \ar[r]\ar[d]&0  \\
\bigoplus_{i\in I}{\rm Hom}_{R}(F_{n}, M_{i})\ar[r]&\bigoplus_{i\in I}{\rm Hom}_{R}(K_{n}, M_{i})\ar[r]&\bigoplus_{i\in I}{\rm Ext}_{R}^{1}(K_{n-1},  M_{i})\ar[r]& 0 \\
}$$
So ${\rm Ext}_{R}^{1}(K_{n-1}, \bigoplus_{i\in I}M_{i})\cong\bigoplus_{i\in I}{\rm Ext}_{R}^{1}(K_{n-1},M_{i})$ and hence by Remark \ref{1.3}(1), $\bigoplus_{i\in I}M_{i}$ is $n$-weak injective if and only if every $M_{i}$ is $n$-weak injective. 
\end{proof}

\begin{prop}\label{1.6} 
Let $n$ be a non-negative integer. Then, the following assertions hold:
\begin{enumerate}
\item[\rm (1)]  A left $R$-module $M$  is $n$-weak injective if and only if every pure submodule and pure
epimorphic image of $M$ is $n$-weak injective.

\item [\rm (2)] A right $R$-module $M$  is $n$-weak flat if and only if every pure submodule and pure
epimorphic image of $M$ is $n$-weak flat.
\end{enumerate}
\end{prop}
\begin{proof}
(1) Let $M$ be an $n$-weak injective left $R$-module and $N$  a pure submodule of $M$.
Then there exists a pure exact sequence $0\rightarrow N\rightarrow M\rightarrow{M}/{N}\rightarrow 0$ which
gives rise to a split exact sequence $0\rightarrow({M}/{N})^{*}\rightarrow M^*\rightarrow N^*\rightarrow 0$ of right $R$-modules. By
Proposition \ref{1.4}(2), $M^*$
is $n$-weak flat. Now by Proposition \ref{1.5}(2), $M^*$ is $n$-weak flat if and only if $N^*$
and $({M}/{N})^*$  are
 $n$-weak flat. Hence by Proposition \ref{1.4}(2), we deduce that $M$ is $n$-weak injective if and only if $N$ and ${M}/{N}$ are $n$-weak injective. Similarly, we prove (2) using Propositions \ref{1.4}(1) and \ref{1.5}(1).
\end{proof}

Now, from the previous results in this section, the following results are obtained.
\begin{thm}\label{1.11}
Any direct product of $n$-weak flat right $R$-modules is $n$-weak flat.
\end{thm}
\begin{proof}
Let $\{M_{i}\}_{i\in I}$ be a family of $n$-weak flat right $R$-modules. By \cite[Proposition 2.4.5]{JX}, $M_{i}$ is pure in $M_{i}^{**}$, hence $\prod_{i\in I}M_{i}$ is pure in $\prod_{i\in I}M_{i}^{**}$. Thus, using Corollary \ref{1.6}(2), it suffices to prove that $\prod_{i\in I}M_{i}^{**}$  is $n$-weak flat. Using \cite[Theorem 2.4]{Rot1}, we have $\prod_{i\in I}M_{i}^{**} \cong  (\bigoplus_{i\in I} M_{i}^{*})^{*}$. 
By Proposition \ref{1.4}(1), each $M_{i}^{*}$ is $n$-weak injective, and so $\bigoplus_{i\in I} M_{i}^{*}$ is $n$-weak injective by Proposition \ref{1.5a}. Thus $(\bigoplus_{i\in I} M_{i}^{*})^{*}$ is $n$-weak flat by Proposition \ref{1.4}(2), and so is $\prod_{i\in I}M_{i}^{**}$.
\end{proof}
\begin{prop}\label{1.7}
 Let $n$ be a non-negative integer. Then,
\begin{enumerate}
\item [\rm (1)] 
If a left $R$-module $M$ is $n$-weak injective, then ${\rm Ext}_{R}^{i}(K_{n-1}, M)=0$ for every  special super finitely presented left $R$-module $K_{n-1}$  and $i\geq 1$.
\item [\rm (2)] 
If a right $R$-module $M$  is $n$-weak flat, then ${\rm Tor}_{i}^{R}(M, K_{n-1})=0$ for every  special super finitely presented left $R$-module $K_{n-1}$  and $i\geq 1$.
\end{enumerate}
\end{prop}
\begin{proof}
(1) Let $U$ be an $n$-super finitely presented left $R$-module with  special super finitely presented module $K_{n-1}$. Then by  Remark \ref{1.3}(3),  $U$ is $(n+1)$-super presented with special super finitely presented module $K_{n}$.  Also, by Remark \ref{1.3}(4), every $n$-weak injective left $R$-module is $(n+1)$-weak injective, and hence 
${\rm Ext}_{R}^{n+2}(U, M)\cong{\rm Ext}_{R}^{1}(K_{n}, M)=0$. There exists an exact sequence $0\rightarrow K_{n}\rightarrow F_{n}\rightarrow K_{n-1}\rightarrow 0$ of left $R$-modules, where $F_n$ is finitely generated and projective. Thus, it follows that ${\rm Ext}_{R}^{1}(K_{n}, M)\cong {\rm Ext}_{R}^{2}(K_{n-1}, M)$ and so ${\rm Ext}_{R}^{2}(K_{n-1}, M)=0$. Repeating this process, we conclude that ${\rm Ext}_{R}^{i}(K_{n-1}, M)=0$ for every $i\geq 1$.

(2) Assume that $U$ is an $n$-super finitely presented left $R$-module with  special super finitely presented module  $K_{n-1}$. By Remark \ref{1.3}(4), every $n$-weak flat right $R$-module  is $(n+1)$-weak flat, and hence ${\rm Tor}_{n+2}^{R}(M,U)\cong{\rm Tor}_{1}^{R}(M, K_{n})=0$. There exists an exact sequence $0\rightarrow K_{n}\rightarrow F_{n}\rightarrow K_{n-1}\rightarrow 0$ of left $R$-modules,  where $F_n$ is finitely generated and projective. We deduce that ${\rm Tor}_{1}^{R}(M, K_{n})\cong {\rm Tor}_{2}^{R}(M, K_{n-1})$, then ${\rm Tor}_{2}^{R}(M, K_{n-1})=0$, and so  ${\rm Tor}_{i}^{R}(M, K_{n-1})=0$ for any $i\geq 1$.
\end{proof}
\begin{cor}\label{1.7a}{\rm (\cite[Propositin 3.1]{Z.W})}
\begin{enumerate}
\item [\rm (1)] 
If a left $R$-module $M$  is weak injective, then ${\rm Ext}_{R}^{i}(F, M)=0$ for every super finitely presented left $R$-module $F$  and $i\geq 1$.
\item [\rm (2)] 
If a right $R$-module $M$  is weak flat, then ${\rm Tor}_{i}^{R}(M, F)=0$ for every  super finitely presented left $R$-module $F$  and $i\geq 1$.
\end{enumerate}
\end{cor}
\section{Special super finitely presented dimension of modules and rings}
\ \ \ In this section, we give some homological aspects of modules with finite $n$-weak
injective and $n$-weak flat dimensions. Let  $n$, $k$  be  non-negative integers. We denote by $\mathcal{WI}_{k}^n(R)$ and $\mathcal{WF}_k^n(R^{op})$ the classes of left and right modules with 
$n$-weak injective dimension and $n$-weak flat dimensions  at most $k$, respectively. If $k=0$, then $\mathcal{WI}_{0}^n(R)=\mathcal{WI}^n(R)$ and $\mathcal{WF}_0^n(R^{op})=\mathcal{WF}^n(R^{op})$, see Remark \ref{1.3}(5). 
\begin{Def}\label{2.1}
The $n$-weak injective dimension of a left module $M$ and $n$-weak flat dimension of a right module $N$
are defined by

$n$-${\rm wid}_{R}(M)$= ${\rm inf}\{k: \ {\rm Ext}_{R}^{k+1}(K_{n-1}, M)=0 \ for \ every \ special \ super \ finitely\ presented\  left\\ module\  K_{n-1}\}$,

and

$n$-${\rm wfd}_{R}(N)$= $ {\rm inf}\{k: \ {\rm Tor}^{R}_{k+1}(N, K_{n-1})=0 \ for \ every \ special \ super \ finitely\ presented\ left\\ module \  K_{n-1}\}$.
\end{Def}
If $k=0$, then by Remark \ref{1.3}(1), it follows that $M$ and $N$ are $n$-weak injective and $n$-weak flat, respectively.

\begin{prop}\label{2.2} 
Let $M$ be a left $R$-module. Then,  the following assertions are equivalent:
\begin{enumerate}
\item [\rm (1)]
$n$-${\rm wid}_{R}(M)\leq k$.
\item [\rm (2)]
${\rm Ext}_{R}^{k+i}(K_{n-1}, M)=0$ for any special super finitely presented left $R$-module $K_{n-1}$  and any $i\geq 1$.
\item [\rm (3)]
${\rm Ext}_{R}^{k+1}(K_{n-1}, M)=0$ for any special super finitely presented left $R$-module $K_{n-1}$.
\item [\rm (4)]
If a sequence of the form $0\rightarrow M\rightarrow E_0 \rightarrow E_1 \rightarrow\cdots\rightarrow E_k\rightarrow0$ is exact with $E_0, E_1,\cdots E_{k-1}$ are
$n$-weak injective modules, then $E_k$ is also $n$-weak injective.
\item [\rm (5)]
There exists an exact sequence  $0\rightarrow M\rightarrow E_0 \rightarrow E_1 \rightarrow\cdots\rightarrow E_k\rightarrow0$, where each $E_i$ is $n$-weak injective.
\end{enumerate}
\end{prop}
\begin{proof}
$(2)\Rightarrow  (3)\Rightarrow (1)$ and $(4)\Rightarrow  (5)$ are obvious.

$(1)\Rightarrow  (2)$ 
We proceed by induction on $k$. If $k=0$, then the result follows by Proposition \ref{1.7}(1). Suppose now the following assertion: if
$n$-${\rm wid}_{R}(M)\leq k$,
then
${\rm Ext}_{R}^{k+i}(K_{n-1}, M)=0$ for any special super finitely presented left $R$-module $K_{n-1}$ and any $i\geq 1$. Let us prove the following assertion: if 
$n$-${\rm wid}_{R}(M)\leq k+1$, then
${\rm Ext}_{R}^{k+1+i}(K_{n-1}, M)=0$ for any special super finitely presented left $R$-module $K_{n-1}$ and any $i\geq 1$. Suppose that $n$-${\rm wid}_{R}(M)\leq k+1$.
If  $n$-${\rm wid}_{R}(M)\leq k$.
then by induction hypothesis we have
${\rm Ext}_{R}^{k+i}(K_{n-1}, M)=0$ for any special super finitely presented left $R$-module $K_{n-1}$ and any $i\geq 1$, and so ${\rm Ext}_{R}^{k+1+i}(K_{n-1}, M)=0$ for any special super finitely presented left $R$-module $K_{n-1}$ and any $i\geq 1$. 
Now, if $n$-${\rm wid}_{R}(M)= k+1$, then ${\rm Ext}_{R}^{k+2}(K_{n-1}, M)=0$ for any special super finitely presented left $R$-module. By induction on $i$, we prove that ${\rm Ext}_{R}^{k+1+i}(K_{n-1}, M)=0$ for any special super finitely presented left $R$-module $K_{n-1}$ and $i\geq1$. The assertion is clearly true for $i=1$.  Suppose now that ${\rm Ext}_{R}^{k+1+i}(K_{n-1}, M)=0$ for any special super finitely presented left $R$-module $K_{n-1}$ and let us prove that ${\rm Ext}_{R}^{k+2+i}(K_{n-1}, M)=0$ for any special super finitely presented left $R$-module $K_{n-1}$.  Let $U$ be an $n$-super presented left $R$-module with special super finitely presented module $K_{n-1}$.
Then, there exists a short exact sequence $0\to K_n\to F_n\to K_{n-1}\to 0$, and we have ${\rm Ext}_{R}^{k+i+1}(K_{n}, M)\cong{\rm Ext}_{R}^{k+i+2}(K_{n-1}, M)$. But $K_0$ is $n$-super finitely presented with special super finitely presented module $K_n$,  hence by induction hypothesis ${\rm Ext}_{R}^{k+i+1}(K_{n}, M)=0$ and consequently ${\rm Ext}_{R}^{k+i+2}(K_{n-1}, M)=0$, as desired.

$(2)\Rightarrow (4)$ Consider the short exact sequence $0\to M\to E_0\to E_0/M\to 0$. Then for any special super finitely presented left $R$-module $K_{n-1}$, we have $$\cdots\to {\rm Ext}_{R}^{k+1}(K_{n-1}, E_0) \to {\rm Ext}_{R}^{k+1}(K_{n-1}, E_0/M)\to {\rm Ext}_{R}^{k+2}(K_{n-1}, M)\to \cdots$$
 But by assumption ${\rm Ext}_{R}^{k+2}(K_{n-1}, M)=0$ and by Proposition \ref{1.7}(1) ${\rm Ext}_{R}^{k+1}(K_{n-1}, E_0)$, so ${\rm Ext}_{R}^{k+1}(K_{n-1}, E_0/M)=0$. Step by step, we prove that ${\rm Ext}_{R}^{1}(K_{n-1}, E_k)=0$

$(5)\Rightarrow  (1)$  Follows from Proposition \ref{1.7}(1).
\end{proof}
\begin{prop}\label{2.3} 
Let $M$ be a right $R$-module. Then,  the following assertions are equivalent:
\begin{enumerate}
\item [\rm (1)]
$n$-${\rm wfd}_{R}(M)\leq k$.
\item [\rm (2)]
${\rm Tor}_{k+i}^{R}(M, K_{n-1})=0$ for any special super finitely presented  left $R$-module $K_{n-1}$  and any $i\geq 1$.
\item [\rm (3)]
${\rm Tor}_{k+1}^{R}(M, K_{n-1})=0$ for any special super finitely presented  left $R$-module $K_{n-1}$.
\item [\rm (4)]
If a sequence of the form $0\rightarrow F_k\rightarrow F_{k-1} \rightarrow\cdots\rightarrow F_1\rightarrow F_0\rightarrow0$ is exact with $F_0, F_1,\cdots F_{k-1}$ are
$n$-weak flat modules, then $F_k$ is also $n$-weak flat.
\item [\rm (5)]
There exists an exact sequence  $0\rightarrow F_k\rightarrow F_{k-1} \rightarrow\cdots\rightarrow F_1\rightarrow F_0\rightarrow0$, where each $F_i$ is $n$-weak flat.
\end{enumerate}
\end{prop}
\begin{proof}
The proof is similar to the proof of Proposition \ref{2.2}, using Proposition \ref{1.7}(2).
\end{proof}

 \begin{cor}\label{2.3a}
Let  $0\rightarrow A\rightarrow B\rightarrow C\rightarrow 0$ be a short exact sequence of $R$-modules. Then, the following assertions hold:
\begin{enumerate}
\item [\rm (1)] 
If $A$ and $ B$ are in $\mathcal{WI}_{ k}^{n}(R)$, then  $C$ is in $\mathcal{WI}_{ k}^{n}(R)$.
\item [\rm (2)]
If $B$ and $C$ are in $\mathcal{WF}_{ k}^{n}(R^{op})$, then $A$ is in $\mathcal{WF}_{ k}^{n}(R^{op})$.
\end{enumerate}
\end{cor}
\begin{proof}
(1) Let $U$ be an $n$-super finitely presented left $R$-module with  special super finitely presented module $K_{n-1}$. If $A$ and $ B$ are in $\mathcal{WI}_{ k}^{n}(R)$, then by Proposition \ref{2.2}, the short exact sequence $0\rightarrow A\rightarrow B\rightarrow C\rightarrow 0$  induces the following exact sequence:
$$0={\rm Ext}_{R}^{k+1}(K_{n-1}, B)\rightarrow {\rm Ext}_{R}^{k+1}(K_{n-1}, C) \rightarrow {\rm Ext}_{R}^{k+2}(K_{n-1}, A)=0.$$
 Hence ${\rm Ext}_{R}^{k+1}(K_{n-1}, C)=0$, and so $C$ is in $\mathcal{WI}_{ k}^{n}(R)$.

(2) Assume that $0\rightarrow A\rightarrow B\rightarrow C\rightarrow 0$ is a short exact sequence of right $R$-modules. Then, by hypothesis and Proposition \ref{2.3}, we have:
$$0={\rm Tor}_{k+2}^{R}(K_{n-1}, C)\rightarrow {\rm Tor}_{k+1}^{R}(K_{n-1}, A) \rightarrow {\rm Tor}_{k+1}^{R}(K_{n-1}, B)=0,$$
where $K_{n-1}$ is a special super finitely presented module associated to an $n$-super presented left $R$-module $U$. Consequently, ${\rm Tor}_{k+1}^{R}(K_{n-1}, A)=0$ and so $A$ is in $\mathcal{WF}_{ k}^{n}(R^{op})$.
\end{proof}

The following proposition is a special case of \cite[Proposition 3.5]{Z.W}.
\begin{prop}\label{2.3l} 
Let $M$ be an $R$-module. Then,  the following assertions are equivalent:
\begin{enumerate}
\item [\rm (1)]
${\rm pd}_{R}(K_{n-1})\leq k$  for all special super finitely presented left $R$-modules $K_{n-1}$.
\item [\rm (2)]
${\rm fd}_{R}(K_{n-1})\leq k$  for all special super finitely presented left  $R$-modules $K_{n-1}$.
\item [\rm (3)]
${\rm Ext}_{R}^{k+1}(K_{n-1}, M)=0$ for any special super finitely presented  left $R$-module $K_{n-1}$  and any left $R$-module $M$.
\item [\rm (4)]
${\rm Tor}_{k+1}^{R}(M, K_{n-1})=0$ for any special super finitely presented  left $R$-module $K_{n-1}$  and any right $R$-module $M$.
\end{enumerate}
\end{prop}
\begin{prop}\label{2.4}
Let $M$ be an $R$-module. Then, the following assertions hold:
\begin{enumerate}
\item [\rm (1)] 
$M$ is in $\mathcal{WF}_{k}^{n}(R^{op})$  if and only  if  $M^{*}$ is in $\mathcal{WI}_{ k}^n(R)$.
\item [\rm (2)] 
 $M$ is in $\mathcal{WI}^n{ k}(R)$  if and only  if  $M^{*}$ is in $\mathcal{WF}_{ k}^{n}(R^{op})$.
\end{enumerate}
\end{prop}
\begin{proof}
(1) We proceed by induction on $k$. If $k=0$, then by Propositions \ref{1.4}(1) and \ref{2.3}, $M$ is $n$-weak flat if and only if $M^*$ is $n$-weak injective. Consider the short exact sequence $0\rightarrow M\rightarrow E\rightarrow L\rightarrow 0,$ where $E$ is injective. Then, $M$ is in $\mathcal{WF}_{ k}^{n}(R^{op})$ if and only if $L$ is in $\mathcal{WF}_{ k-1}^{n}(R^{op})$. We have $L$ is in $\mathcal{WF}_{ k-1}^{n}(R^{op})$ if and only if $L^*$ is in $\mathcal{WI}_{ k-1}^n(R)$ and so $M^*$ is in $\mathcal{WI}_{ k}^n(R)$. Similarly, by Propositions \ref{1.4}(2) and \ref{2.2}, (2) holds.
\end{proof}
\begin{prop}\label{2.6}
Let $M$ be an $R$-module. Then, the following assertions hold:
\begin{enumerate}
\item [\rm (1)] 
$M$ is in $\mathcal{WI}_{ k}^{n}(R)$ if and only if  for every pure submodule $N$ of $M$, $N$ and $ {M}/{N}$ are in $\mathcal{WI}_{ k}^{n}(R)$.
\item [\rm (2)] 
$M$ is in $\mathcal{WF}_{ k}^{n}(R^{op})$ if and only if  for every pure submodule $N$ of $M$, $N$ and $ {M}/{N}$ are in $\mathcal{WF}_{ k}^{n}(R^{op})$.
\end{enumerate}
\end{prop}
\begin{proof}
Similar to the proof of Proposition \ref{1.6} by using 
Proposition \ref{2.4}.
\end{proof}
\section{Covers and preenvelopes by modules with $n$-weak injective and $n$-weak
flat dimension at most $k$}
\ \ Let $\Y$ be a class of $R$-modules and $M$ be an 
 $R$-module. Following \cite{EM}, we say that a  morphism $f : F\rightarrow M$ is a
$\Y$-precover of $M$ if $F\in\Y$ and ${\rm Hom}_{R}(F^{'}, F) \rightarrow {\rm Hom}_{R}(F^{'},M)\rightarrow 0$ is exact
for all $F^{'}\in\Y$. Moreover, if whenever a  morphism $g : F\rightarrow F$ such that
$fg = f$ is an automorphism of $F$, then $f : F\rightarrow M$ is called an $\Y$-cover of $M$. The
class $\Y$ is called (pre)covering if each $R$-module has a $\Y$-(pre)cover. Dually,
the notions of $\Y$-preenvelopes, $\Y$-envelopes and (pre)enveloping classes are defined.

A duality pair over $R$ \cite{HJ} is a pair $(\mathcal{M}, \mathcal{C})$, where $\mathcal{M}$ is a class of
left $R$-modules and $\mathcal{C}$ is a class of right $R$-
modules, subject to the following conditions:
(1) For an $R$-module $M$, one has $M\in \mathcal{M}$ if and only if $M^{*}\in \mathcal{C}$.
(2) $\mathcal{C}$ is closed under direct summands and finite direct sums.\\
A duality pair $(\mathcal{M}, \mathcal{C})$ is called (co)product-closed if the class $\mathcal{M}$ is closed under
(co)products in the category of all left $R$-modules.
A duality pair $(\mathcal{M}, \mathcal{C})$ is called perfect if it is coproduct-closed, $\mathcal{M}$ is closed
under extensions, and $R$ belongs to $\mathcal{M}$.

 Let $\X$ be a class of $R$-modules. We denote by $\mathscr{I}(R)$ the class of injective left modules and by $\mathscr{P}(R)$ the class of projective right modules. We call $\X$ injectively resolving \cite{HH} if $\mathscr{I}(R)\subseteq \X$, and for every short exact sequence
$0\rightarrow A\rightarrow B\rightarrow C\rightarrow0$ with $A\in\X$, $B\in\X$ if and only if $C\in\X$. Also, we call $\X$ projectively resolving if $\mathscr{P}(R)\subseteq \X$, and for every short exact sequence
$0\rightarrow A\rightarrow B\rightarrow C\rightarrow0$ with $C\in\X$, $A\in\X$ if and only if $B\in\X$.

In this section,  by the use of duality pairs, we investigate $\mathcal{WI}_{k}^{n}(R)$ and $\mathcal{WF}_{k}^{n}(R^{op})$ as preenveloping and covering classes

\begin{prop}\label{2.12}
The pair $(\mathcal{WI}_{ k}^{n}(R), \mathcal{WF}_{ k}^{n}(R^{op}))$ is a duality pair.
\end{prop}
\begin{proof}
Let $\{M_{i}\}_{i\in I}$ be a family of right $R$-modules. If every $M_i$ is in $\mathcal{WF}_{ k}^{n}(R^{op})$, then we claim that $\bigoplus_{i\in I}M_i$ is in $\mathcal{WF}_{ k}^{n}(R^{op})$. By induction, if $k=0$, then by Proposition \ref{1.5}(2), $\bigoplus_{i\in I} M_{i}$ is $n$-weak flat.  For every $R$-module $M_i$, there exists a short exact sequence $0\rightarrow L_i\rightarrow P_i\rightarrow M_i\rightarrow 0$  of right $R$-modules, where  $P_i$ is  projective. Thus we have the following short exact sequence $0\rightarrow \bigoplus_{i\in I} L_i\rightarrow \bigoplus_{i\in I}P_i\rightarrow \bigoplus_{i\in I}M_i\rightarrow 0$. Since each $L_i$ is in $\mathcal{WF}_{ k-1}^{n}(R^{op})$, by the induction hypothesis, we have that $\bigoplus_{i\in I}L_i$ is in $\mathcal{WF}_{ k-1}^{n}(R^{op})$, and so it follows that $\bigoplus_{i\in I}M_i$ is in $\mathcal{WF}_{k}^{n}(R^{op})$.
Also by Proposition \ref{2.4}(2), $M$ is in $\mathcal{WI}_{ k}^{n}(R)$ if and only if $M^{*}$ is in $\mathcal{WF}_{ k}^{n}(R^{op})$. On the other hand, we have $\mathscr{P}(R)\subseteq \mathcal{WF}_{ k}^{n}(R^{op})$, and by Corollary \ref{2.3a}(2), we deduce that 
the class $\mathcal{WF}_{ k}^{n}(R^{op})$ is projectively resolving. Then by  \cite[Proposition 1.4]{HH}, the class $\mathcal{WF}_{ k}^{n}(R^{op})$
 is closed under direct summands, and so we conclude that $(\mathcal{WI}_{ k}^{n}(R), \mathcal{WF}_{ k}^{n}(R^{op}))$ is a duality pair.
\end{proof}

\begin{prop}\label{2.13}
The  pair $(\mathcal{WF}_{ k}^{n}(R^{op}), \mathcal{WI}_{ k}^{n}(R))$ is a duality pair.
\end{prop}
\begin{proof}
By Proposition \ref{2.4}(1), $M$ is in $\mathcal{WF}_{ k}^{n}(R^{op})$ if and only if $M^{*}$ is in $\mathcal{WI}_{ k}^{n}(R)$. Let $\{M_{i}\}_{i\in I}$ be a family of left $R$-modules. Suppose that each $M_i$ is in $\mathcal{WI}_{ k}^{n}(R)$ and let us show that $\prod_{i\in I}M_i$ is in $\mathcal{WI}_{ k}^{n}(R)$. By induction, if $k=0$, then by Proposition \ref{1.5}(1), $\prod_{i\in I} M_{i}$ is $n$-weak injective. There exists a short exact sequence $0\rightarrow M_i\rightarrow E_i\rightarrow D_i\rightarrow 0$  of left $R$-modules, where  $E_i$ is  injective, and  consequently, we have the following short exact sequence $0\rightarrow \prod_{i\in I} M_i\rightarrow \prod_{i\in I}E_i\rightarrow \prod_{i\in I}D_i\rightarrow 0$. Since each $D_i$ is in $\mathcal{WI}_{ k-1}^{n}(R)$, by induction hypothesis, we deduce that $\prod_{i\in I}D_i$ is in $\mathcal{WI}_{ k-1}^{n}(R)$, and it follows easily that $\prod_{i\in I}M_i$ is in $\mathcal{WI}_{ k}^{n}(R)$. So in particular, every finite direct sum of family $\{M_i\}_{i\in I}$ in $\mathcal{WI}_{ k}^{n}(R)$ is in $\mathcal{WI}_{k}^{n}(R)$.
Also, we have $\mathscr{I}(R)\subseteq \mathcal{WI}_{ k}^{n}(R)$, and by Corollary \ref{2.3a}(1), it is clear that the class  $\mathcal{WI}_{ k}^{n}(R)$ is injectively resolving. Then, by  \cite[Proposition 1.4]{HH}, the class $\mathcal{WI}_{ k}^{n}(R)$
 is closed under direct summands, and hence, it follows that $(\mathcal{WF}_{ k}^{n}(R^{op}), \mathcal{WI}_{ k}^{n}(R))$ is a duality pair.
\end{proof}
\begin{lem}\label{1.14}{\rm (\cite[Theorem 3.1]{HJ})}
Let $(\mathcal{M}, \mathcal{C})$ is a duality pair. Then $\mathcal{M}$ is closed under pure submodules, pure quotients, and pure extensions. Furthermore, the following hold:
\begin{enumerate}
\item [\rm (1)] 
 If $(\mathcal{M}, \mathcal{C})$ is product-closed then $\mathcal{M}$ is preenveloping.
\item [\rm (2)] 
If $(\mathcal{M}, \mathcal{C})$ is coproduct-closed then $\mathcal{M}$ is covering.
\item [\rm (3)] 
If $(\mathcal{M}, \mathcal{C})$ is perfect then $(\mathcal{M}, \mathcal{M}^{\bot})$ is a perfect cotorsion pair.
\end{enumerate}
\end{lem}
\begin{thm}\label{2.14}
The class $\mathcal{WI}_{ k}^{n}(R)$ is  covering and preenveloping.
\end{thm}
\begin{proof}
By Proposition \ref{2.6}(1), the class $\mathcal{WI}_{ k}^{n}(R)$ is closed under pure submodules, pure quotients, and pure extensions. Also, by the proof of the Proposition \ref{2.13}, the class $\mathcal{WI}_{ k}^{n}(R)$ is closed under direct products, and similarly by using  Proposition \ref{1.5a}, we see that the class $\mathcal{WI}_{ k}^{n}(R)$ is also closed under direct sums. By Proposition \ref{2.12}, the pair  $(\mathcal{WI}_{ k}^{n}(R), \mathcal{WF}_{ k}^{n}(R^{op}))$  is a duality pair, and so from Lemma \ref{1.14} follows that the class $\mathcal{WI}_{ k}^{n}(R)$ is  covering and preenveloping.
\end{proof}
\begin{thm}\label{2.16}
The class $\mathcal{WF}_{ k}^{n}(R^{op})$ is  covering and preenveloping.
\end{thm}
\begin{proof}
By Proposition \ref{2.6}(2), the class $\mathcal{WF}_{ k}^{n}(R^{op})$ is closed under pure submodules, pure quotients, and pure extensions. Now, we show that the class $\mathcal{WF}_{ k}^{n}(R^{op})$ is closed under direct products. 
If $\{M_{i}\}_{i\in I}$  is a family of right $R$-modules, where $M_i$ is in $\mathcal{WF}_{ k}^{n}(R^{op})$ for any $i\in I$, then by induction, if $k=0$, by Theorem \ref{1.11}, it follows that $M_i$ is $n$-weak flat if and only if $\prod_{i\in I}M_i$ is $n$-weak flat. There exists a short exact sequence $0\rightarrow L_i\rightarrow P_i\rightarrow M_i\rightarrow 0$  of right $R$-modules, where  $P_i$ is  projective. Then, there exists the short exact sequence $0\rightarrow \prod_{i\in I} L_i\rightarrow \prod_{i\in I}P_i\rightarrow \prod_{i\in I}M_i\rightarrow 0$. If $M_i$ is in $\mathcal{WF}_{ k}^{n}(R^{op})$, then $L_i$ is in $\mathcal{WF}_{ k-1}^{n}(R^{op})$. By induction hypothesis, $\prod_{i\in I}L_i$ is in $\mathcal{WF}_{ k-1}^{n}(R^{op})$, and so
we conclude that $\prod_{i\in I}M_i$ is in $\mathcal{WF}_{ k}^{n}(R^{op})$. Similarly, using Proposition \ref{1.5}(2), we see that the class $\mathcal{WF}_{ k}^{n}(R^{op})$ is closed under direct sums.
  Since the pair $(\mathcal{WF}_{ k}^{n}, \mathcal{WI}_{ k}^{n})$  is a duality pair by Proposition \ref{2.13}, we conclude that the class $\mathcal{WF}_{ k}^{n}(R^{op})$  is  covering and preenveloping from Lemma \ref{1.14}.
\end{proof}

Notice that $\mathcal{WI}_{ k}^{0}(R)= \mathcal{WI}_{ k}(R)$ and $\mathcal{WF}_{ k}^{0}(R^{op})=\mathcal{WF}_{ k}(R^{op})$. Indeed, by Remark \ref{1.3}(1), $\mathcal{WI}_{ k}(R)$
and $\mathcal{WF}_{ k}(R^{op})$
are the classes of left  $R$-modules and right $R$-modules with
weak injective dimension and weak flat dimension  at most $k$, respectively, see \cite{.NG}. Thus
by  Proposition \ref{2.4} and Theorems \ref{2.14} and \ref{2.16},  we have the following result:

\begin{cor}\label{1.14t}{\rm (\cite[Theorems 4.4, 4.5, 4.8 and 4.9]{.NG})}
\begin{enumerate}
\item [\rm (1)] 
The class $\mathcal{WI}_{k}(R)$ is  covering and preenveloping.
\item [\rm (2)] 
The class $\mathcal{WF}_{k}(R^{op})$ is  covering and preenveloping.
\end{enumerate}
\end{cor}
Also, if $k=0$, then from Proposition \ref{2.4} and Theorems \ref{2.14} and \ref{2.16} we have:
\begin{cor}\label{1.20t}{\rm (\cite[Theorem 2.5]{JZW})}
The following assertions hold:
\begin{enumerate}
\item [\rm (1)] 
Every left (resp. right) $R$-module has an $fp_n$-injective (resp. $fp_n$-flat) cover.
\item [\rm (2)] 
Every left (resp. right) $R$-module has an $fp_n$-injective (resp. $fp_n$-flat) preenvelope.
\item [\rm (3)] 
If $M\rightarrow N$ is an $fp_n$-injective (resp. $fp_n$-flat) preenvelope of a left (resp. right) $R$-module $M$, then $N^*\rightarrow M^*$ is an $fp_n$-flat (resp. $fp_n$-injective) precover of $M^*$.
\end{enumerate}
\end{cor}

Now we give some equivalent characterizations for $_RR$ to be in $\mathcal{WI}_{ k}^{n}(R)$ in terms of properties of  $\mathcal{WI}_{ k}^{n}(R)$ and  $\mathcal{WF}_{ k}^{n}(R^{op})$.
\begin{prop}\label{2.17}
The following assertions are equivalent:
\begin{enumerate}
\item [\rm (1)]
$_RR$ is in $\mathcal{WI}_{ k}^{n}(R)$.
\item [\rm (2)]
Every right $R$-module has a monic $\mathcal{WF}_{ k}^{n}(R^{op})$-preenvelope.
\item [\rm (3)]
Every  injective right $R$-module is in $\mathcal{WF}_{ k}^{n}(R^{op})$.
\item [\rm (4)]
Every flat left $R$-module is in $\mathcal{WI}_{ k}^{n}(R)$.
\item [\rm (5)]
Every projective left $R$-module is in $\mathcal{WI}_{ k}^{n}(R)$.
\item [\rm (6)]
Every  left $R$-module  has an epic $\mathcal{WI}_{ k}^{n}(R)$-cover.
\end{enumerate}
\end{prop}
\begin{proof}
$(1)\Rightarrow  (2)$ By Theorem \ref{2.16}, every  $R$-module $M$  has a $\mathcal{WF}_{ k}^{n}(R^{op})$-preenvelope $f: M\rightarrow F$. By Proposition \ref{2.4}(2), $R^*$ is in $\mathcal{WF}_{ k}^{n}(R^{op})$, and so similar to proof the of Theorem \ref{2.16}, one can prove that $\prod_{i\in I} R^{*}$ is in $\mathcal{WF}_{ k}^{n}(R^{op})$. Also, $(_RR)^*$ is a cogenerator. So
we have the following exact sequence  $0\rightarrow M\stackrel{\displaystyle g}\rightarrow \prod_{i\in I} R^{*}$, and hence there exists a morphism $h: F\to \prod_{i\in I} R^{*}$ such that $hf=g$  and so $f$ is monic.

$(2)\Rightarrow  (3)$ Let $E$ be an injective right $R$-module. By assumption, let $f: E\rightarrow F$ be a  monic $\mathcal{WF}_{ k}^{n}(R^{op})$-preenvelope of $E$. Therefore, the split exact sequence $0\rightarrow E\rightarrow F\rightarrow{F}/{E}\rightarrow 0$ exists, and so $E$ is a direct summand of $F$. Hence by Proposition \ref{2.3} and \cite[Proposition 7.6]{Rot2},  $E$  is in $\mathcal{WF}_{ k}^{n}(R^{op})$.

$(3)\Rightarrow  (1)$ By assumption,  $R^*$ is in  $\mathcal{WF}_{ k}^{n}(R^{op})$, since $R^*$ is injective. So $_RR$ is  in $\mathcal{WI}_{ k}^{n}(R)$ by Proposition \ref{2.4}(2).

$(3)\Rightarrow  (4)$
Let $F$ be a flat left $R$- module. Then by \cite[Theorem 3.52]{Rot1}, $F^*$ is injective and so $F^*$ is in $\mathcal{WF}_{ k}^{n}(R^{op})$ by assumption, and hence $F$ is in $\mathcal{WI}_{ k}^{n}(R)$ by Proposition \ref{2.4}(2).

$(4)\Rightarrow  (5)$  and $(5)\Rightarrow  (1)$ are  clear.

$(6)\Rightarrow  (1)$ By assumption, $_RR$  has an epimorphism  $\mathcal{WI}_{ k}^{n}(R)$-cover $f:D\rightarrow R$.
   Then we have a split exact sequence $0\rightarrow {\rm Ker}f \rightarrow D\rightarrow R\rightarrow 0$ with $D$ is
in $\mathcal{WI}_{ k}^{n}(R)$. Then, by Proposition \ref{2.2} and also in particular by \cite[Proposition 7.22]{Rot2}, $_RR$ is in $\mathcal{WI}_{ k}^{n}(R)$.

$(1)\Rightarrow  (6)$
First, we show that if  $\{M_{i}\}_{i\in I}$  is a family of left $R$-modules and $M_i$ is in $\mathcal{WI}_{ k}^{n}(R)$, then $\bigoplus_{i\in I}M_i$ is in $\mathcal{WI}_{ k}^{n}(R)$. If $k=0$, then by Proposition \ref{1.5a}, $\bigoplus_{i\in I} M_{i}$ is $n$-weak injective if and only if each $M_i$ is  $n$-weak injective.  The short exact sequence $0\rightarrow M_i \rightarrow E_i\rightarrow N_i\rightarrow 0$, where $E_i$ is injective exists. Consequently, the sequence 
 $0\rightarrow \bigoplus_{i\in I}M_i \rightarrow  \bigoplus_{i\in I}E_i\rightarrow  \bigoplus_{i\in I}N_i\rightarrow 0$ is also exact. If $M_i$ is in $\mathcal{WI}_{ k}^{n}(R)$, then by Proposition \ref{2.2}, $N_i$ is in $\mathcal{WI}_{ k-1}^{n}(R)$.  By induction hypothesis, $\bigoplus_{i\in I}N_i$ is in $\mathcal{WI}_{ k-1}^{n}(R)$ and then by  Proposition \ref{2.2} again, $\bigoplus_{i\in I}M_i$ is in $\mathcal{WI}_{ k}^{n}(R)$. On the other hand, by Theorem \ref{2.14}, 
 there is a $\mathcal{WI}_{ k}^{n}(R)$-cover $\psi: X \rightarrow M$ for any left $R$-module $M$. Also, there is an exact sequence $0\rightarrow K \rightarrow P\stackrel{\displaystyle h}\rightarrow M\rightarrow 0$ of left $R$-modules, where $P$ is an $R$-module free. Since $_RR$ is in $\mathcal{WI}_{ k}^{n}(R)$, then it follows that $P=\bigoplus_{i\in I}R$ is $\mathcal{WI}_{ k-1}^{n}(R)$.  So there exists a map $g: P \rightarrow X $ such that $\psi g=h$. Since $h$ is epic, we deduce that $\psi: X \rightarrow M$ is also  epic.
\end{proof}
Now we define the $n$-super finitely presented dimension of a ring which generalizes a homological dimension introduced and studied in \cite{Z.G} defined as ${\rm l.sp. gldim}(R)={\rm sup}\{{\rm pd}_{R}(M) \mid M \ \rm{is \ a \ super \ finitely \ presented \ left \ module}\}.$
\begin{Def}\label{2.1w}
 Let $n$ be a non-negative integer. Define\\
${\rm l.nsp. gldim}(R):={\rm sup}\{{\rm pd}_{R}(K_{n-1}) \mid K_{n-1} \ \text{is a  special\  super \ finitely presented left module} \}.$
\end{Def}
It is clear that ${\rm l.n.sp. gldim}(R)\leq{\rm l.sp. gldim}(R)$ for any $n\geq 0$. If $n=0$, then ${\rm l.n.sp. gldim}(R)={\rm l.sp. gldim}(R)$.
In examples \ref{1.1a} and \ref{1.3a}, since $R$ is coherent, we have ${\rm l.sp. gldim}(R)=2$ by \cite[Theorem 3.8]{Z.W}. But, ${\rm l.n.sp. gldim}(R)\leq1$ for any $n\geq 1$, since ${\rm pd}_{R}(U)\leq 2$ for every $n$-super finitely presented left $R$-module $U$.
\begin{prop}\label{1.17a}
The following assertions are equivalent:
\begin{enumerate}
\item [\rm (1)]
Every right $R$-module has an epic $\mathcal{WF}_{ k}^{n}(R^{op})$-envelope.
\item [\rm (2)]
$M$ is in $\mathcal{WI}_{ k+1}^{n}(R)$ for every  left $R$-module $M$.
\item [\rm (3)]
$N$ is in $\mathcal{WF}_{ k+1}^{n}$ for every right  $R$-module $N$.
\item [\rm (4)]
Every  $R$-module  has a monic $\mathcal{WI}_{ k}^{n}(R)$-cover.
\item [\rm (5)]
Every quotient of any $n$-weak injective left $R$-module is in  $\mathcal{WI}_{ k}^{n}(R)$.
\item [\rm (6)]
 Every submodule of any $n$-weak flat right $R$-module is in $\mathcal{WF}_{ k}^{n}(R^{op})$.
\item [\rm (7)]
The kernel of any $\mathcal{WI}_{ k}^{n}(R)$-precover of any left $R$-module is in $\mathcal{WI}_{ k}^{n}(R)$.
\item [\rm (8)]
The cokernel of any $\mathcal{WF}_{ k}^{n}(R^{op})$-preenvelope of any right $R$-module is in $\mathcal{WF}_{ k}^{n}(R^{op})$.
\item [\rm (9)]
${\rm l.nsp. gldim}(R)\leq k+1$.
\end{enumerate}
\end{prop}
\begin{proof}
$(1)\Leftrightarrow (6)$ 
Consider the class $\mathcal{WF}_{k}^{n}(R^{op})$ of modules with 
$n$-weak flat dimension at most $k$. Then, similar to the proofs of the Proposition \ref{2.12} and Theorem \ref{2.16}, the class $\mathcal{WF}_{k}^{n}(R^{op})$ is closed under direct summands and direct products, respectively. So \cite[Theorem 2]{CD} shows that (1) and (6) are equivalent.

$(4)\Leftrightarrow (5)$ 
Consider the class $\mathcal{WI}_{k}^{n}(R)$ of left modules with 
$n$-weak injective dimension at most $k$. Then, similar to the proofs of the Propositions \ref{2.13} and  \ref{2.17}($(1)\Rightarrow  (6)$), the class $\mathcal{WI}_{k}^{n}(R)$ is closed under direct summands and direct sums, respectively. 
Thus from \cite[Proposition 4]{Z.JJG}, it follows that (4)  and (5) are equivalent.

$(6)\Rightarrow  (5)$
Let $N$ be a submodule of $n$-weak injective left $R$-module $M$. Then, there exists a short exact sequence $ 0\rightarrow N\rightarrow M\rightarrow {M}/{N}\rightarrow 0$ which induces the
exactness of $ 0\rightarrow ({M}/{N})^*\rightarrow M^*\rightarrow N^*\rightarrow 0$. By Proposition \ref{1.4}(2),  $M^*$  is $n$-weak flat right $R$-module, and hence by hypothesis, $ ({M}/{N})^*$ is in $\mathcal{WF}_{ k}^{n}(R^{op})$. Consequently, using Proposition \ref{2.4}(2), we conclude that ${M}/{N}$ is in $\mathcal{WI}_{ k}^{n}(R)$.

$(5)\Rightarrow  (6)$ Similar to the proof of $(6)\Rightarrow  (5)$ using Propositions  \ref{1.4}(1 ) and \ref{2.4}(1).

$(1)\Rightarrow  (8)$ 
Let $M$ be a right $R$-module. Then by Theorem \ref{2.16}, there is a $\mathcal{WF}_{ k}^{n}(R^{op})$-preenvelope.  $\psi : M\rightarrow D$. Also by hypothesis,  if  the map $\phi: M\rightarrow Y$ is an epic $\mathcal{WF}_{ k}^{n}(R^{op})$-envelope of $M$, then from \cite[Lemma 8.6.3]{EM}, it follows that $L\oplus Y\cong D$, where $L={\rm Coker}\psi$. So  $L$ is in $\mathcal{WF}_{ k}^{n}(R^{op})$ as a direct summand of $D$.

$(8)\Rightarrow  (6)$  Consider the short exact sequence $ 0\rightarrow L\rightarrow M\rightarrow D\rightarrow 0$ of right $R$-modules, where $M$ is $n$-weak flat and $L$  a submodule of $M$. We claim that $L$ is in $\mathcal{WF}_{ k}^{n}(R^{op})$. Indeed, we have the following commutative diagram:
$$\xymatrix{
0\ar[r]& L\ar[r]\ar@{=}[d]&M\ar[d]&  \\
&L\ar[r]^{h}&X\ar[r]&Y\ar[r]& 0\\
}$$
where $h: L\rightarrow X$ is a $\mathcal{WF}_{ k}^{n}(R^{op})$-preenvelope of $L$ and $Y={\rm Coker}h$. In particular, the sequence $ 0\rightarrow L\rightarrow X\rightarrow Y\rightarrow 0$ is exact, and so by Corollary \ref{2.3a}(2), $L$ is  in $\mathcal{WF}_{ k}^{n}(R^{op})$.

$(5)\Rightarrow  (2)$ For every left $R$-module $M$, there is an exact sequence 
$ 0\rightarrow M\rightarrow E\rightarrow D\rightarrow 0$ of left $R$-modules, where $E$ is injective. By (5), $D$ is in $\mathcal{WI}_{ k}^{n}(R)$ and so by Proposition \ref{2.2},  $M$ is in $\mathcal{WI}_{ k+1}^{n}(R)$.

$(2)\Rightarrow  (5) $ Clear.

$(2)\Leftrightarrow (3)\Leftrightarrow (9)$ Clear  by Proposition \ref{2.3l}. 
\end{proof}

When $k=0$, we have the following result.

\begin{thm}\label{2.1oo}
The following assertions are equivalent:
\begin{enumerate}
\item [\rm (1)]
$_RR$ is in $\mathcal{WI}^{n}(R)$.
\item [\rm (2)]
Every left $R$-module is in $\mathcal{WI}^{n}(R)$.
\item [\rm (3)]
Every special super finitely presented left $R$-module is in $\mathcal{WI}^{n}(R)$.
\item [\rm (4)]
The short exact sequence $0\rightarrow K_{n}\rightarrow F_{n}\rightarrow K_{n-1}\rightarrow0$ is a split superpure sequence.
\item [\rm (5)]
Every right $R$-module is in $\mathcal{WF}^{n}(R^{op})$.
\end{enumerate}
\end{thm}
\begin{proof}
$(2)\Rightarrow  (3)$ and  $(2)\Rightarrow  (1)$ are trivial.

$(1)\Rightarrow  (2)$ 
Let $N$ be a left $R$-module. Consider $P\rightarrow N\rightarrow 0$ where $P$ is free. Since $_RR$ is in $\mathcal{WI}^{n}(R)$, by Proposition \ref{1.5a}, we get  that $P$ is in $\mathcal{WI}^{n}(R)$, and so by Proposition \ref{1.17a}, $N$ is in 
$\mathcal{WI}^{n}(R)$.

$(3)\Rightarrow  (4)$ Let $U$ be an $n$-super finitely presented left $R$-module with  special super finitely presented module $K_{n-1}$. Then, we have the short exact sequence $0\rightarrow K_{n}\rightarrow F_{n}\rightarrow K_{n-1}\rightarrow0$ left $R$-modules. Since $U$ is also
$(n+1)$-super finitely presented, it follows that $K_{n}$ is special super finitely presented. By assumption, $K_{n}$ is in $\mathcal{WI}^{n}(R)$ and thus by Remark \ref{1.3}(1), ${\rm Ext}_{R}^{1}(K_{n-1}, K_{n})=0$ and so by Proposition \ref{1.10}, the above sequence is split superpure.

$(4)\Rightarrow  (5)$  Let the short exact sequence $0\rightarrow K_{n}\rightarrow F_{n}\rightarrow K_{n-1}\rightarrow0$ be split superpure. Then $K_{n-1}$ is flat as a direct summand of $F_{n}$. Consequently, ${\rm Tor}_{1}^{R}(M, K_{n-1})=0$ for any right $R$-module $M$, and so by Remark \ref{1.3}(1), $M$ is in $\mathcal{WF}^{n}(R^{op})$.

$(5)\Rightarrow  (2)$
Let $M$ be any left $R$-module. Then,  $M^*$ is a right $R$-module and hence by assumption, $M^*$ is in $\mathcal{WF}^{n}(R^{op})$. Therefore by Proposition \ref{1.4}(2), every  left $R$-module $M$ is $\mathcal{WI}^{n}(R)$.
\end{proof}

A cotorsion pair (or orthogonal theory of ${\rm Ext}$) consists of a pair $(\mathcal{F}, \mathcal{C})$ of
classes of $R$-modules \cite {SA,SAT} such that $\mathcal{C}= \mathcal{F}^{\bot}$ and $\mathcal{F} =^{\bot}\mathcal{C}$ where for a class $\mathcal{S}$, we have
$\mathcal{S}^{\bot}=\{M : M \ {\rm is \ an} \  R$-${\rm module \ and} \ {\rm Ext}_{R}^{1}(S,M) = 0 \ {\rm for \ all} \ S\in\mathcal{S}\}$
and $^{\bot}\mathcal{S}=\{M : M \ {\rm is \ an} \  R$-${\rm module \ and} \ {\rm Ext}_{R}^{1}(M,S)=0 \ {\rm for \ all} \ S\in\mathcal{S}\}$.

A cotorsion theory $(\F, \mathcal{C})$ is called hereditary,  if whenever $0 \rightarrow F^{'}\rightarrow F\rightarrow F^{''}\rightarrow 0$ is exact  with $F, F^{''}\in \F$ then $F^{'}$ is also in $\F$, or equivalently, if $0 \rightarrow C^{'}\rightarrow C\rightarrow C^{''}\rightarrow 0$ is an exact sequence with $C, C^{'}\in \mathcal{C}$, then $C^{''}$ is also in $\mathcal{C}$. A
cotorsion pair $(\F, \mathcal{C})$ is called complete provided that for any $R$-module $M$,
there exist exact sequences $0\rightarrow M\rightarrow C\rightarrow D\rightarrow 0$ and $0\rightarrow C^{'}\rightarrow D^{'}\rightarrow M\rightarrow 0$
of $R$-modules with $C, C^{'}\in\mathcal{C}$ and $D, D^{'}\in\mathcal{F}$, for more details, see \cite{ Z.JG,NG}.

\begin{prop}\label{1.17u}
\begin{enumerate}
\item [\rm (1)]
If $n$-${\rm wid}_{R}(R)\leq k$, then the pair
$(\mathcal{WI}_{k}^n(R), \mathcal{WI}_{k}^n(R)^{\bot})$
is a perfect cotorsion pair.
\item [\rm (2)]
The pair $(^{\bot}\mathcal{WI}_{k}^n(R), \mathcal{WI}_{k}^n(R))$ is a hereditary cotorsion pair.
\end{enumerate}
\end{prop}
\begin{proof}
(1) The pair $(\mathcal{WI}_{ k}^{n}(R), \mathcal{WF}_{ k}^{n}(R^{op}))$ is a duality pair by Proposition \ref{2.12}. Similar to the proof of the Proposition \ref{2.12}, we show that $\mathcal{WI}^n(R)$ is closed under direct sums and under extensions. Also by hypothesis and Proposition \ref{2.3}, $R$ is in $\mathcal{WI}_{k}^n(R)$, and so $(\mathcal{WI}_{k}^n(R), \mathcal{WF}_{k}^{n}(R^{op}))$ is a perfect duality pair. So by Lemma \ref{1.14}, it follows that $(\mathcal{WI}_{k}^n(R), \mathcal{WI}_{k}^n(R)^{\bot})$ is a perfect cotorsion pair.

(2) First, we show that $(^{\bot}\mathcal{WI}_{k}^{n}(R))^{\bot}=\mathcal{WI}_{k}^n(R)$. It is clear that $\mathcal{WI}_{k}^n(R)\subseteq (^{\bot}\mathcal{WI}_{k}^n(R))^{\bot}$. Let $M$ is in $(^{\bot}\mathcal{WI}_{k}^n(R))^{\bot}$ and $U$ be an $n$-super finitely presented left $R$-module with special super finitely presented module $K_{n-1}$. Then, it follows that $K_{n-1}$ is in ${^{\bot}\mathcal{WI}_{k}^n(R)}$ and consequently, ${\rm Ext}_{R}^{1}(K_{n-1}, M)=0$. Thus by Remark \ref{1.3}(1), ${\rm Ext}_{R}^{n}(U, M)=0$ and hence
$M$ is in $\mathcal{WI}_{k}^n(R)$. Let $0 \rightarrow M_1\rightarrow M_2\rightarrow M_3\rightarrow 0$ be a short exact sequence left $R$-modules such that $M_{1}$ and $M_{2}$ are in $\mathcal{WI}_{k}^n(R)$. Then from Corollary \ref{2.3a}(1), we conclude that $M_{3}$ is in  $\mathcal{WI}_{k}^n(R)$, and so   $(^{\bot}\mathcal{WI}_{k}^n(R), \mathcal{WI}_{k}^n(R))$ is a hereditary cotorsion pair.
\end{proof}
\begin{prop}\label{1.18}
The pair $(\mathcal{WF}_{k}^n(R^{op}), \mathcal{WF}_{k}^n(R^{op})^{\bot})$ is a hereditary perfect cotorsion pair.
\end{prop}
\begin{proof}
The pair $(\mathcal{WF}_{k}^n(R^{op}), \mathcal{WI}_{k}^n(R))$ is a duality pair by Proposition \ref{2.13}. The class $\mathcal{WF}_{k}^n(R^{op})$ is closed under direct sums and extensions. By Remark \ref{1.3}(5), $R$ is in $\mathcal{WF}_{k}^n(R^{op})$, and hence $(\mathcal{WF}_{k}^n(R^{op}), \mathcal{WI}_{k}^n(R))$ is a perfect duality pair. So by Lemma \ref{1.14},   $(\mathcal{WF}_{k}^n(R^{op}), \mathcal{WF}_{k}^n(R^{op})^{\bot})$ is a  perfect cotorsion pair. Let $0 \rightarrow M_1\rightarrow M_2\rightarrow M_3\rightarrow 0$ be a short exact sequence of right $R$-modules such that $M_{2}$ and $M_{3}$ are in $\mathcal{WF}_{k}^{n}(R^{op})$. Then from Corollary \ref{2.3a}(2), we get that $M_{1}$ is in $\mathcal{WF}_{k}^n(R^{op})$, and so   $(\mathcal{WF}_{k}^n(R^{op}), \mathcal{WF}_{k}^n(R^{op})^{\bot})$  is a hereditary cotorsion pair.
\end{proof}
Let $(\mathcal{A}, \mathcal{B})$ and $(\mathcal{C}, \mathcal{D})$ be two cotorsion pairs. Then by \cite[Remark 4.12]{NG}, $(\mathcal{A}, \mathcal{B})\preceq (\mathcal{C}, \mathcal{D})$ if $\mathcal{B}\subseteq\mathcal{D}$.  By \cite[Definition 6.1]{HW}, the pair $(\mathcal{M}, \mathcal{N})$ is said to be cogenerated by a set if there is a set
of objects $M\in\mathcal{M}$ such that $N\in\mathcal{N}$ if and only if ${\rm Ext}_{R}^{1}(M,N)=0$ for all
 $M\in\mathcal{M}$. 
 In \cite{TE}, Eklof and Trlifaj proved that a cotorsion pair $(\mathcal{F}, \mathcal{C})$ in $R$-Mod is complete
 when it is cogenerated by a set. This result actually holds in any Grothendieck
 category with enough projectives, as Hovey proved in \cite{HW}.
Then by Remark \ref{1.3}, we have the following easy observations:
\begin{rem}\label{1.19}
\begin{enumerate}
\item [\rm (1)]
Let $\mathcal{S}Pres_{n}^{\infty}$ be a subclass of all the
special $n$-super finitely presented  left $R$-modules. Then,  $(^{\bot}\mathcal{WI}^n(R), \mathcal{WI}^n(R))$ is a  hereditary complete cotorsion pair, since it is cogenerated
by a set of representatives for $\mathcal{S}Pres_{n}^{\infty}$.
\item [\rm (2)]
There is a serie of  hereditary complete cotorsion pairs for any $n\geq 0$ and $k\geq 0$ as follows:
$$(^{\bot}\mathcal{WI}_{k}^n(R), \mathcal{WI}_{k}^n(R))\preceq (^{\bot}\mathcal{WI}_{k}^{n+1}(R), \mathcal{WI}_{k}^{n+1}(R))\preceq (^{\bot}\mathcal{WI}_{k}^{n+2}(R), \mathcal{WI}_{k}^{n+2}(R))\preceq\cdots$$
\item [\rm (3)]
There is a serie of  hereditary cotorsion pairs  for any $n\geq 0$ and $k\geq 0$ as follows:
$$(\mathcal{WF}_{k}^n(R^{op}), \mathcal{WF}_{k}^n(R^{op})^{\bot})\preceq(\mathcal{WF}_{k}^{n+1}(R^{op}), \mathcal{WF}_{k}^{n+1}(R^{op})^{\bot})\preceq\cdots$$
\end{enumerate}
\end{rem}

\bigskip

\noindent\textbf{Acknowledgment.}  The authors
would like to thank the referee for the helpful suggestions and valuable comments. 
\bigskip


\end{document}